\documentclass[12pt]{article}
\usepackage[T1]{fontenc}
\usepackage[utf8]{inputenc}
\usepackage[english]{babel}

\usepackage{graphicx}
\usepackage{float}
\usepackage{hyperref}
\usepackage[bottom=2cm,top=2cm,left=2cm,right=2cm]{geometry}
\usepackage{amssymb, amssymb, amsmath, amsfonts, mathrsfs, amsthm}
\newtheorem{definition}{Definition}
\newtheorem{theorem}{Theorem}
\usepackage{xcolor}

\usepackage{fancybox}

\usepackage{mathpazo}

\numberwithin{equation}{section}

\newcommand{\Y}{\mathsf{Y}}

\newcommand{\D}{\mathcal{D}}

\newcommand{\W}{\mathsf{R}}

\newcommand{\Frak}{\mathfrak{F}}

\newcommand{\dint}{\displaystyle\int}
\newcommand{\Uo}{\mathsf{S}}

\theoremstyle{plain}
\newtheorem{thm}{Theorem}[section]

\newtheorem{lem}[thm]{Lemma}

\newtheorem{remark}[thm]{Remark}

\numberwithin{equation}{section}

\renewenvironment{proof}{$\mathbf{Proof}$.}{\begin{flushright}\qedsymbol\end{flushright}}
\title{\textsc{Controllability of nonautonomous measure driven integrodifferential evolution equations with nonlocal conditions}}
\author{ 
	Mamadou Niang $^{a}$ \footnote{E-mail address: niang.mamadou@ugb.edu.sn},
 \:Mamadou Pathé LY $^{a}$ \footnote{E-mail address: ly.mamadou-pathe@ugb.edu.sn},
	\:Abdoul Aziz Ndiaye $^{a}$\footnote{E-mail address: ndiaye.abdoul-aziz5@ugb.edu.sn},\\
	\:Mamadou Abdoul Diop $^{a,b}$ \footnote{E-mail address: mamadou-abdoul.diop@ugb.edu.sn}
}

\begin{document}
	\newcommand\norm[1]{\left\lVert#1\right\rVert}
	\renewenvironment{proof}{$\mathbf{Proof}$.}{\begin{flushright}\qedsymbol\end{flushright}}
	\newenvironment{prob}{$\mathbf{Problem}$.}{\emph{\begin{flushright}\qedsymbol\end{flushright}}}
	\newenvironment{problem}[1]{\par\noindent\underline{Problem:}\space#1}{}
	\providecommand{\keywords}[1]{\textbf{\textit{Keywords:}} #1}
	\providecommand{\classification}[1]{\textbf{\textit{Mathematical Subject Classification:}} #1}
	\renewenvironment{proof}{$\mathbf{Proof}$.}{\begin{flushright}\qedsymbol\end{flushright}}
	\selectlanguage{english}
	\maketitle
	{\footnotesize
		\begin{flushleft}			
			\thanks{$^a$ Universit\'e Gaston Berger de Saint-Louis, \, UFR SAT D\'epartement de Math\'ematiques, B.P234,\,Saint-Louis, S\'en\'egal}\\
			\thanks{$^b$ UMMISCO UMI 209 IRD/UPMC, Bondy, France}
	\end{flushleft}}
	\begin{abstract}
	This research delves into the exact  controllability of semilinear measure-driven integrodifferential systems in nonlocal settings. We give enough controllability requirements using the measure of noncompactness and the Mönch fixed point theorem without making any assumptions about how compact the evolution system is in relation to the linear part of the measure  system. We find results here that both generalize and improve upon many prior findings.
\end{abstract}	
	\\
	\classification{	
		26A42,
		34A38,
		34K30,
		34K35,
		93B05.}	
\section{Introduction}
Measure differential equations, sometimes called measure-driven equations, help to represent dynamical systems with discontinuous trajectories. They may model results that contradict classical physics rules, such as Zeno trajectories and the effects of quantum mechanics. However, the main distinction between this type of equation and the more common impulsive differential equation is that the former permits unlimited discontinuous points within a finite time interval. For more details, check out \cite{goebel2009hybrid, kronig1931quantum, lygeros2008hybrid}.
 \par According to Byszewski's works \cite{byszewski1991theorems}, nonlocal initial conditions are more effective at addressing physical concerns than the traditional initial condition $u(0)=u_0$. Deng et al. \cite{deng1993exponential} looked at how a small amount of gas moved through a clear tube and used the nonlocal conditions $u(0)=\sum_{k=1}^{n}c_ku(t_k)$ to get more accurate results. More precise measurements at $t_k, k=1,2,\ldots,m$ are possible with the initial condition $u(0)=\sum_{k=1}^{m}c_ku(t_k)$ as opposed to just measuring at $t=0$.
\par However, control theory relies on controllability. Using the allowed controls to go from any starting state to any given ending state makes a dynamic control system controllable. Because of its centrality in control theory, its importance in physics, and its practical implications in engineering, it attracts the attention of many mathematicians and engineers. Various scientific and technical sectors have recently shown a strong interest in controllability (for more reading, see \cite{dieye2017controllability,fu2006controllability,diop2021existence}).
\par Beyond that, it's important to note that most previous work focused on scenarios where the operators had no bearing on the passage of time.  When discussing specific parabolic evolution issues, the operators rely on the time variable $t$. It is common for applications to employ this category of operators \cite{Acq, Amann, Chen, Pchen, Weiss}. Recent studies have focused on improving measure-driven evolution systems. The research conducted by Cao and Sun \cite{caojitao2017controllability} investigated the   exact controllability for the following  evolution system:
\begin{equation}\label{sysref}
\left\{\begin{array}{l c l}
\mathrm{d}x(t) = \Big[\mathcal{A}(t)x(t)  + Bu(t)\Big]\mathrm{d}t + f(t,x(t))\mathrm{d}g(t),\ t\in J,\\ 
x(0) + p(x)=x_0,&\\
\end{array}\right.
\end{equation}
where $J=[0,b]$ with $b>0$. The state variable $x(\cdot)$ takes values in Banach space $X$ with the norm $\|\cdot\|$. $\mathcal{A}(t)$ is a family of linear operators which generates an evolution system $\{\mathcal{U}(t,s):0\le s\le t\le b\}$. 
\par Based on those earlier studies, this study looks into the controllability of  mild solutions  for the following semilinear measure-driven integrodifferential system with nonlocal conditions: \begin{equation}\label{sys}
\left\{\begin{array}{l c l}
\mathrm{d}\zeta(t)= \left[\mathcal{A}(t)\zeta(t) + \displaystyle\int_{0}^{t}\Delta(t,s)\zeta(s)\mathrm{d}s + \mathcal{V}u(t)\right]\mathrm{d}t  + \delta(t,\zeta(t))\mathrm{d}h(t),\ t\in I,\\
\zeta(0) + g(\zeta)=\zeta_0.&\\
\end{array}\right.
\end{equation}
where $I=[0,a], a>0$. The state variable $\zeta(\cdot)$ takes values in a Banach space $\mathbb{X}$ endowed with the norm $\|\cdot\|$. $\mathcal{A}(t)\, : \mathsf{D}(\mathcal{A})\;\subset \mathbb{X}\to\mathbb{X}$ is a family of a closed linear operator with a fixed domain $\mathsf{D}(\mathcal{A})$. $\Delta(t,s)\;:\mathsf{D}(\Delta)\;\subset \mathbb{X}\to\mathbb{X}$ is a closed linear operator with domain $\mathsf{D}(\Delta)\supset\mathsf{D}(\mathcal{A})$ for each pair $(t,s)$ with $0\le s\le t <\infty$. The nonlinear function $\delta: I\times\mathbb{X}\to\mathbb{X}$ and nonlocal function $g:\mathcal{R}(I, \mathbb{X})\to \mathbb{X}$ will be specified later, where $\mathcal{R}(I, \mathbb{X})$ denotes the space of all regulated functions on $I$ in which we will consider the problem.
The control $u(\cdot)$ belongs to $L^2(I, \mathbb{K})$, which is a Banach space of admissible control functions with $\mathbb{K}$ a Banach space. $\mathcal{V}:\mathbb{K}\to\mathbb{X}$ is a bounded linear operator. $h:I\to\mathbb{R}$ is a left non-decreasing discontinuous function, and $\mathrm{d}h$ denotes the distributional derivative of the function $h$.
We must investigate the controllability of the measure-driven system in regulated functions, as it differs from the spaces of continuous and piecewise continuous functions. The impulses in measure-driven equations, which depend on the measure $h$, are intrinsic and possibly note prefixed. We are aware of no previous reports on the measure-driven evolution of integrodifferential systems under nonlocal circumstances \eqref{sys}. This work aims to fill this knowledge gap by addressing the controllability of the aforementioned measure-driven integrodifferential  system  (\ref{sys}).
The solution to autonomous evolution equations is considerably more straightforward than non-autonomous ones, as is well known. We take our inspiration from \cite{diop2022integrodifferential}. The operator families establish relationships between the resolvent operator and the evolution family. Using $\mathcal{A}(t)$ to make sure that the system has a mild solution is possible, and this can be done even if the resolvent operator $\W(t, s)$ for the linear part of the system (\ref{sys}) is not compact. As a result, we prove our findings by utilizing the noncompactness measure and the M\"onch fixed point theorem. Based on the method proposed by Grimmer \cite{Grimmer}, we will assume that the norm continuity of the resolvent operator is equal to the norm continuity of the evolution family. The work that we have done can be seen as building upon and expanding upon the ideas presented in \cite{caojitao2017controllability} and similar works.
 Below are the key points, significance, and originality of this article:
\begin{enumerate}
	\item A mild solution to the system (\ref{sys}) is established for the first time by applying an integral equation given in terms of the resolvent operator. Then, we obtain some sufficient conditions to ensure system controllability by using the theory of evolution family, the measure of noncompactness and M\"onch fixed point theorem .
	\item The significance to study the MDEs is that one can model Zeno trajectories because \textcolor{blue}{$h$}
as a function of bounded variation may exhibit infinitely many discontinuity points in
a finite interval. Such systems arise in game theory, non-smooth mechanics, and other
systems \cite{Lygeros12, leo27}.
	\item As the authors of \cite{Sure10, Gu22} state, MDEs are pretty challenging to deal with because they are less continuous or smooth than ordinary differential equations. Since MDEs are only continuous and confined, future studies will be more challenging.
\item The findings from our derivation expand upon the prior findings of Cao and Sun \cite{caojitao2017controllability}.
\item  Lastly, we provide an application to back up the study's validity.
\end{enumerate}
The following is an outline of our paper's content. We devote the section \ref{sect2} of this paper to introducing and reviewing some fundamental notations, preliminary concepts, and lemmas that will be useful throughout the rest of the work. The solvability and controllability of the nonlinear system will focus our attention in Section \ref{sect3}. Section \ref{sect4} presents an example to illustrate our primary findings.
\section{Preliminaries}\label{sect2}
\subsection{Integrodifferential equations}
Grimmer's resolvent operator plays a central role in solving integrodifferential equations of the type we are studying here \eqref{sys}. To recall the basic properties of this resolvent operator, we'll make a few additional assumptions.
Let $\Y$ be a Banach space formed from $\mathsf{D}(\mathcal{A})$ with graph norm  
\begin{displaymath}
\|y\|_\Y:=\|\mathcal{A}(0)y\|_\mathbb{X} +\|y\|_\mathbb{X}\;\text{ for }\, y\in \Y.
\end{displaymath}
As $\mathcal{A}(t)$ and $\Delta(t,s)$ are closed linear operators, it implies that $\mathcal{A}(t)$ and $\Delta(t,s)$ are in the set of bounded operators from $Y$ to $\mathbb{X}$, $\mathsf{B}(\Y, \mathbb{X}),$ for $0\le t\le a$ and $0\le s\le t\le a$, respectively. We further suppose that $A(t)$ and $\Delta(t,s)$ are continuous on  $0\le t\le a$ and $0\le s\le t\le a$, respectively, into $\mathsf{B}(\Y, \mathbb{X})$. Let's now investigate the subsequent homogeneous linear equation : 
\begin{equation}\label{VE}
\left\lbrace
\begin{array}{rl}
\zeta'(t) &= \mathcal{A}(t)\zeta(t) +\displaystyle\int_0^t \Delta(t, s)\zeta(s)ds\quad for\; t \in[0,a],\\
\zeta(0) &= \zeta_0 \in \mathbb{X}.
\end{array}
\right.
\end{equation} 
\begin{definition}\cite{Grimmer}\label{DefResolvant}. \\
	A family of bounded linear operators on $\mathbb{X}$ with the form $\{\W(t, s), 0 \leq s \leq t\}$ is referred to as a resolvent operator for Eq. \eqref{VE}, if the following properties are true :
	\begin{itemize}
		\item[$(\mathsf{i})$] $\W(t, s)$ is strongly continuous in $s$ and $t$, $\W(s, s) = id_{ \mathbb{X}}$ (the identity map of $ \mathbb{X}$), $0 \leq s \leq t\le a$ and $\|\W(t, s)\| \leq M e^{\gamma(t -s)}$ for some constants $M$ and $\gamma$, 
		\item[$(\mathsf{ii})$] $\W(t, s)\Y \subset \Y$ and $\W(t, s)$ is strongly continuous in $s$ and $t$ on $\Y$, 
		\item[$(\mathsf{iii})$] For every $\zeta \in \Y$, $\W(t, s)\zeta$ is strongly continuously differentiable in $s$ and $t$ and nonempty
		\begin{equation}\label{partial}
		\begin{cases}
		\dfrac{\partial \W}{\partial t} (t, s)\zeta = \mathcal{A}(t)\W(t, s)\zeta + \displaystyle\int_s^t\Delta(t, r)\W(r, s)\zeta\mathrm{d}r, \\
		\dfrac{\partial \W}{\partial s }(t, s)\zeta = -\W(t, s)\mathcal{A}(s)\zeta -\displaystyle\int_s^t \W(t, r)\Delta(r, s)\zeta dr. 
		\end{cases}
		\end{equation}
		with $\dfrac{\partial \W }{\partial t}(t,s)\zeta$ and $\dfrac{\partial \W }{\partial s}(t,s)\zeta$ strongly continuous on $0\leq s\leq t$.
	\end{itemize}
\end{definition}
With the eventual resolvent operator, we will use this latter fact to define a mild solution to \eqref{sys} (see below).
\begin{lem}\cite{Grimmer}\label{Autonome}
	If $\mathcal{A}(t)\equiv \mathcal{A}$ and $\Delta(t,s)\equiv \Delta(t-s)$ and there exists a resolvent operator $\W(t,s)$, then $\W(t,s)=\W(t-s,0)$. 
\end{lem}
\begin{definition}
	Let $(\mathcal{A}(t))_{t\in[0,a]}$ be a family of  infinitesimal generators of $C_0$-semigroups. $(\mathcal{A}(t))_{t\in[0,a]}$ is called stable if  there exist constants $M_0\geq 1$ and $\alpha_0$ such that 	
	\begin{displaymath}
	\left\|\prod_{k=1}^n(\mathcal{A}(s_k)-\lambda id_{\mathbb{X}})^{-1}\right\|_{\mathcal{L}(\mathbb{X})} \leq \dfrac{M_0}{(\lambda -\alpha_0 )^{n}}
	\end{displaymath}	
	for all $\lambda > \alpha_0,$ and for every finite sequence $ 0\leq s_1 \leq s_2 \leq \ldots \leq s_n,\, n=1,2,3,\ldots$
\end{definition}
Let $\Frak=BUC(\mathbb{R}^+;\mathbb{X})$ be the space of all bounded uniformly continuous functions defined from $\mathbb{R}^+$ to $\mathbb{X}$. Let $\Delta(t),t\ge 0$ be defined as the linear operator from $\Y$ into $\mathfrak{F}$ and given by  $(\Delta(t)y)(s)=\Delta(t+s, t)y$ for $s\ge 0$ and $y\in \Y$. Let $\mathsf{D}_s$ be the differentiation operator defined on $\mathfrak{F}$ by $\mathsf{D}_s\ell=\ell^\prime$ on a domain $\mathcal{D}(\mathsf{D}_s)\subset\mathfrak{F}$. Then, $\mathsf{D}_s$ is the infinitesimal generator of the translation semigroup $(T(t))_{t\ge 0}$ defined on $\mathfrak{F}$ by 
$$
(T(t)\ell)(s)=\ell(t+s)\ \text{ for } \ell\in\Frak\, \text{ and } t,s\ge 0.
$$
Due to Grimmer's work \cite{Grimmer}, with the help of the above notations, we make the following assumptions in order to prove the existence of the resolvent operator of the equation \eqref{VE}.\\

\begin{enumerate}
	\item[$\textbf{(C1)}$]\, $(\mathcal{A}(t))_{t \in[0, a]}$ is a stable family of generators such that $\mathcal{A}(t)\zeta$ is strongly continuously differentiable on $\mathbb R^+$ for $\zeta \in\Y$. 
	Additionally, for $\zeta\in \Y$, $\Delta(t)\zeta$ is strongly continuously differentiable on $\mathbb R^+$.
	\item[$\textbf{(C2)}$]\, $\Delta(t)$ is continuous on   $[0,+\infty[$ into $\mathsf{B}(\Y,\Frak)$,  
	\item[$\textbf{(C3)}$]\, $\Delta(t):\Y\to \D(\mathsf{D}_s)$ for all $t\ge 0$ and  
	 $D_s\Delta(t)$ is continuous on $\mathbb{R}^+$ into $\mathsf{B}(\Y, \Frak)$. 
\end{enumerate}  
If all three of the above conditions are met, the condition $(\mathbf{R})$ is said to be satisfied.
\begin{remark}
	Under the conditions $\mathcal{A}(t)\equiv\mathcal{A}$ and $\Delta(t,s)\equiv\Delta(t-s)$, we say that \textcolor{blue}{$\Delta(t)\in \mathsf{B}(\Y, \Frak)$} is constant, which implies that conditions  \textbf{(C1)-(C3)} are satisfied if $\mathcal{A}$ generates a $C_0$-semigroup.
\end{remark}
\begin{lem}\cite{Grimmer}\label{ER} If the statement $(\mathbf{R})$ is true, then Eq. \eqref{VE} admits a unique resolvent operator. 
\end{lem}
Let's consider the following evolution integrodifferential equation:
\begin{equation}\label{nhomoeq}
\left\lbrace
\begin{array}{rl}
\zeta^\prime(t) &= \mathcal{A}(t)\zeta(t) +\displaystyle\int_0^t \Delta(t, s)\zeta(s)ds + h(t)\quad for\; t\ge 0,\\
\zeta(0) &= \zeta_0 \in \mathbb{X},
\end{array}
\right.
\end{equation} 
where $h\in \mathbb{L}_{loc}(\mathbb{R}^+, \mathbb{X})$.
\begin{definition}\cite{Grimmer} A function $\zeta : [0, +\infty[\to\mathbb{X}$ is said to be a strict solution of \eqref{nhomoeq}, if 
	\begin{enumerate}
		\item[(i)] $\zeta\in\mathcal{L}([0, +\infty);\mathbb{R})\cap\mathcal{L}([0, +\infty); \mathsf{D}(\mathcal{A}))$,
		\item[(ii)] $\zeta$ satisfies the following equation:
		\begin{align}\label{strict}
		\zeta(t)=\W(t,0)\zeta_0+\int_{0}^{t}{\W(t,s)h(s)ds},\, \quad t\in\mathbb{R}^+. 
		\end{align}
		where $\{\W(t,s): 0\le s\le t\le a\}$ is resolvent operator of equation \eqref{VE}. 
	\end{enumerate} 
\end{definition}
The following theorem ensures the existence of the strict solution to equation \eqref{nhomoeq}.
\begin{theorem}\cite{Grimmer}
	Assume that condition $(\mathbf{R})$ is satisfied. If $\zeta_0\in\mathsf{D}(\mathcal{A})$ and $h\in\mathcal{C}^1(\mathbb{R}^+, \mathbb{X})$, then equation \eqref{VE} has a unique strict solution given by \eqref{strict}.
\end{theorem}
\begin{definition}
    A function $\zeta: [0, +\infty[\to\mathbb{X}$ satisfying \eqref{strict} is called a mild solution of equation \eqref{nhomoeq}
\end{definition}
According to definition  \ref{DefResolvant} and the uniform boundedness principle, $$1 \leq M_a:=\displaystyle\sup_{0\leq s\leq t\leq a}\|\W(t,s)\|<\infty.$$   
\subsection{Norm-continuity of the resolvent operators}
In the follow-up, we create an adequate requirement to obtain the norm continuity of the resolvent operator for the system \eqref{VE}. Consider the subsequent Cauchy problem:
\begin{equation}\label{equ2}
\left\lbrace
\begin{array}{lcl}
\zeta^\prime(t) &=& \mathcal{A}(t)\zeta(t),\qquad 0\leq s\leq t\leq a,\\
\zeta(s) &=& l\in\mathbb{X}.
\end{array}
\right. 
\end{equation}
Assume the following conditions on the operator $(\mathcal{A}(t))_{t\in [0, a]}$, taken from \cite[chapter 5, Theorem 6.1]{pazy2012semigroups} 
(and also in \cite{Nagel}).
\begin{enumerate}
	\item[$\textbf{(E1)}$]\, $\mathcal{A}(t)$ is closed and the domain $\mathsf{D}(\mathcal{A}(t))=\mathsf{D}(\mathcal{A})$ is independent of $t$ and is dense in $\mathbb{X}$.
	\item[$\textbf{(E2)}$]\, For each $t\ge 0$, the resolvent $\W(\lambda, \mathcal{A})=(\lambda I + \mathcal{A}(t))^{-1}$ exists for all $\lambda\in\mathbb{C}$ with $\mathsf{R}e\lambda\le 0$, and there exists $M_1>0$ such that $$\|\W(\lambda, \mathcal{A}(t))\|_{\mathcal{L}(\mathbb{X})}\le\frac{M_1}{(|\lambda| + 1)}.$$
	\item[$\textbf{(E3)}$]\, There exists $0<\theta\le 1$ and $M_2>0$ such that 
	$$\|(\mathcal{A}(t)-\mathcal{A}(s))\mathcal{A}(r)^{-1}\|\le M_2|t - s|^\theta\, \text{ for all }\, t,s,r\in [0, a].$$ 
\end{enumerate}  
\begin{lem}\cite[chapter 5, Theorem 6.1]{pazy2012semigroups}
	Assume that \textbf{(E1)-(E3)} holds. Then, there exists a unique evolution system $\{\Uo(t,s):0\leq s\leq t\le a\}$ generated by the family $\{\mathcal{A}(t): t\ge 0\}$.
\end{lem}
When discussing the existence of mild solutions for non-autonomous systems, the theory of evolution family becomes indispensable.

The following is the definition of the evolution system. 
\begin{definition}
	A family $\{\Uo(t,s):0\leq s\leq t\}$ of linear, bounded operators on a Banach space $\mathbb{X}$ is called an evolution family for \eqref{equ2} if
	\begin{description}
		\item[(i)] $\Uo(t,t)=I$ and $\Uo(t,s)=\Uo(t,r)\Uo(r,s)$ for every $0\leq s\leq r\le t\le a$,
		\item[(ii)] The mapping $\{(t,s)\in\mathbb{R}^2: s\leq t \}\ni(t,s)\mapsto\Uo(t,s)$ is strongly continuous, 
		\item[(iii)] $\|\Uo(t,s)\|\leq \mathsf{N}\exp(\omega(t-s))$ for some $\mathsf{N}\ge 1,\omega\in\mathbb{R}$ and all $t\ge s\in\mathbb{R}$,
		\item[(iv)] For each $z\in \Y$, $\Uo(t,s)$ is strongly continuously differentiable in $t$ and $s$ with 
		\begin{equation*}
		\begin{array}{lll}
		(\partial/\partial \,t)\, \Uo(t,s)z&=& \mathcal{A}(t)\Uo(t,s)z,\\\\
		(\partial/\partial \,s )\,\Uo(t,s)z&=& -\Uo(t,s)\mathcal{A}(t)z.
		\end{array}  
		\end{equation*}
		
	\end{description}
\end{definition}

\begin{lem}\cite{Nagel}
	Consider the family of linear operators $(\mathcal{A}(t))_{t\in [0,\infty)}$ on a Banach space $\mathbb{X}$ and the corresponding non-autonomous Cauchy problem \eqref{equ2}. Then, the following assertions are equivalent. 
	\begin{description}
		\item[(i)] Eq. \eqref{equ2} is well posed (with exponentially bounded solutions). 
		\item[(ii)] There is a unique strongly continuous (exponentially bounded) evolution family $\{\Uo(t,s):0\leq s\leq t\le a \}$ on $\mathbb{X}$ solving Eq. \eqref{equ2}. 
	\end{description}
\end{lem}
As $(\mathcal{A}(t))_{t\in[0,\infty)}$ is a stable family of generators and $\mathcal{A}(t)\zeta$ is strongly continuously differentiable on $[0,+\infty)$, so Eq. \eqref{equ2} has a strict solution \cite[p.120]{tanabe}.\\ \\
The following section takes into account the perturbation of equation \eqref{equ2} : 
\begin{equation}\label{equ3}
\left\lbrace
\begin{array}{lcl}
\zeta^\prime(t)&=& \mathcal{A}(t)\zeta(t)+\dint_s^t\Delta(t,u)\zeta(u)du,\qquad 0\leq s\leq t\\
\zeta(s)&=&l\in\mathbb{X}.
\end{array}
\right. 
\end{equation}
Combining the resolvent operator of \eqref{equ3} with the variation of constants formula linked to \eqref{equ2} provides a clear relationship between the resolvent operator and the evolution family. See Diop et al. \cite{diop2022integrodifferential} for more details. 
\begin{definition}
	The resolvent operator $\{\W(t,s):0\leq s\leq t\le a\}$ is said to be norm-continuous if the function $(t,s)\longmapsto \W(t,s)$ is continuous by operator norm for $0\leq s<t$. 
\end{definition}


\begin{theorem}\cite[Theorem 2.9]{diop2022integrodifferential}\label{NCR}
	Assume that \textbf{(R)} and \textbf{(E1)-(E3)} are valid. Let $\left\{\W\left(t,s\right):\ 0\leq s\leq t\le a\right\}$ be the unique resolvent operator of equation \eqref{VE} and $\Uo\left(t,s\right)$ the unique evolution system generated by the family $A(t),\, t\ge  0$. Then, the resolvent operator
	$\W\left(t,s\right)$ for equation \eqref{equ3} is norm-continuous for $T>t-s>0$  if the evolution family $\Uo\left(t,s\right)$ is norm-continuous for $T>t-s>0.$
\end{theorem}
We end this section by the following useful lemma introduced by Ezzinbi et al.\cite{Ezz}.
\begin{lem}\label{uni}
    Let $\mathbb{X}$ be a Banach space and $(T_n)_{n\geq 1}$ be a sequence of bounded linear maps \textcolor{blue}{on $\mathbb{X}$}. Assume that $T_n\zeta\rightarrow T\zeta$ for all $\zeta\in\mathbb{X}$ and for some $T\in\mathcal{L}(\mathbb{X})$. Then, for any compact set $K$ in $\mathbb{X}$, $T_n$ converges to $T$ uniformly in $K$, namely,
	$$
	\displaystyle\sup_{\zeta\in K}\lVert T_n(\zeta)-T(\zeta)\rVert\rightarrow 0,\;\;\mbox{ as }\;\; n\rightarrow +\infty.
	$$
\end{lem}
\subsection{ Equi-regulated functions and measure of noncompactness}
\par This section begins with a review of equiregulated functions. Next, we examine the Hausdorff measure of noncompactness and related facts. Finally, the system's (\ref{sys}) mild solution is expressed.\\
Let $\mathbb{X}$ be a Banach space endowed with the norm $\|\cdot\|$ and $I=[0, a]$ a closed interval of the real line. 
\begin{definition}\cite{mesquita2012measure}
	A function $f:I\to \mathbb{X}$ is called regulated on $I$, if the limits
	$$
	\lim_{s\to t^-} f(s) = f(t^-),\ t\in (0, a]\ \text{and}\ \lim_{s\to t^+} f(s) = f(t^+), \ t\in [0, a)
	$$ 
	exist and are finite.  
\end{definition}
\noindent Throughout this paper $\mathcal{R}(I,\mathbb{X})$ denote the space composed of all regulated functions $f:I\to \mathbb{X}$. Also, it is obvious that the set of discontinuities of a regulated function is at most countable and that the space of regulated functions $\mathcal{R}(I,\mathbb{X})$ is a Banach space endowed with the norm $\|f\|_\infty = \sup_{t\in I}\|f(t)\|$  (see \cite{honig1980volterra} for more details). Consider $L^2(I, \mathbb{K})$ as a Banach space of all $\mathbb{K}$-valued Bochner square-integrable functions defined on $I$ endowed of the norm 
$$
\|\zeta\|_{L^2}=\bigg(\dint_0^a\|\zeta(t)\|_{\mathbb{K}}^2\mathrm{d}t\bigg)^{\frac{1}{2}},\, \zeta\in L^2(I, \mathbb{K}).
$$
\begin{lem}\label{lemme1}
	Let $f:I\to \mathbb{X}$ and $h:I\to\mathbb{R}$ be functions such that $h$ is regulated and $\int_0^a f\mathrm{d}h$ exists. Then for every $t_0\in [0, a]$, the function $$p(t)=\int_{t_0}^{t}f\mathrm{d}h,\, t\in [0, a],$$ is regulated and satisfies 
	\begin{align*}
	p(t^+) = p(t) + f(t)\mathbb{V}^+ h(t), t\in [0, a),\\
	p(t^-) = p(t) - f(t)\mathbb{V}^- h(t), t\in (0, a],
	\end{align*}
	where $\mathbb{V}^+h(t)=h(t^+) - h(t)$ and $\mathbb{V}^-h(t)=h(t) - h(t^-)$. $h(t^-)$ and $h(t^+)$ denote the function's left limit and right limit $h$ at the time $t$, respectively. 
\end{lem}
\begin{definition}\cite{mesquita2012measure}\label{def22}
	A set $\mathcal{G}\subset \mathcal{R}(I, \mathbb{X})$ is called equiregulated if for every $\alpha > 0$ and $t_0\in I$, there is a $\mu > 0$ such that: 
	\begin{enumerate}
		\item[(i)] if $v\in\mathcal{G}, t^\prime\in I$ and $t_0 - \mu<t^\prime<t_0$, then $||v(t_0^-) - v(t^\prime)||<\alpha.$
		\item[(ii)] if $v\in\mathcal{G}, t^\prime\in I$ and $t_0<t^\prime<t_0 + \mu$, then $||v(t^\prime)-v(t_0^+)||<\alpha$.
	\end{enumerate}
\end{definition}
\begin{lem}\cite{mesquita2012measure}\label{lemm23}
	Let $\left\{v_n\right\}_{n=1}^\infty$ be a sequence of functions from $I$ to $\mathbb{X}$. If $v_n$ converges pointwise to $v_0$ as $n\to \infty$ and the sequence $\left\{v_n\right\}_{n=1}^\infty$ is equiregulated, then $v_n$ converges uniformly to $v_0$.
\end{lem}
\begin{lem}\cite{cao2016existence}\label{lemme 2.10}
	Let $\mathcal{V}\subset\mathcal{R}(I, \mathbb{X})$. If $\mathcal{V}$ is bounded and equiregulated, then the set $\overline{co}(\mathcal{V})$ is also bounded and equiregulated.
\end{lem}
\noindent Let $\mathbb{X}$ be a Banach space . The Hausdorff measure of noncompactness of a bounded subset $\mathcal{U}$ of $\mathbb{X}$ is defined as the infimum of the set of all real numbers $\alpha >0$ such that $\mathcal{U}$ can be covered by a finite number of balls of radius smaller than $\alpha$, that is,
$$
\lambda(\mathcal{U})=\inf\{\alpha >0: \mathcal{U}\subset\displaystyle\cup_{i=1}^n\mathcal{B}(\eta_i, r_i), \eta_i\in\mathbb{X}, r_i<\alpha (i=1,\cdots , n), n\in\mathbb{N}\},
$$
where $\mathcal{B}(\eta_i, r_i)$ denotes the open ball centered at $\eta_i$ and of radius $r_i$. 
\begin{lem}\cite{banas2014sequence}
	Let $\mathcal{U}, V$ be bounded subsets of $\mathbb{X}$ and $\kappa\in \mathbb{R}$. Then the following properties are satisfied.
	\begin{enumerate}
		\item[(i)] $\lambda(\mathcal{U})=0$ if and only if $\mathcal{U}$ is relatively compact;
		\item[(ii)] $\mathcal{U}\subseteq V$ implies $\lambda(\mathcal{U})\le \lambda(V)$; 
		\item[(iii)] $\lambda(\mathcal{U})= \lambda(\overline{\mathcal{U}})$
		\item[(iv)] $\lambda(\mathcal{U} + V) \leq \lambda(\mathcal{U}) + \lambda(V)$, where $\mathcal{U} + V = \left\{v| v= x + y: x\in \mathcal{U}, y\in V\right\}$;
		\item[(v)] $\lambda(\mathcal{U}\cup V)=\max\left\{\lambda(\mathcal{U}),\lambda(V)\right\}$;
		\item[(vi)] $\lambda(\kappa\mathcal{U}) = |\kappa|\lambda(\mathcal{U})$ for any $\kappa\in\mathbb{R}$;
		\item[(vii)] $\lambda(co(\mathcal{U}))=\lambda(\mathcal{U})$;
		\item[(viii)] If the map $\psi:\mathsf{D}(\psi)\subseteq \mathbb{X}\to Z$ is lipschitz continuous with a constant $k$, then $\lambda_Z(\psi\Omega)\leq k\lambda(\Omega)$ for any bounded subset \textcolor{blue}{$\Omega\in\mathsf{D}(\psi)$}, where $Z$ is a Banach space.
	\end{enumerate}
\end{lem}
\noindent Let $\mathcal{V}$ be a subset of $\mathcal{R}(I, \mathbb{X})$. For any fixed $t\in I$, we denote $\mathcal{V}(t)=\{v(t):v\in \mathcal{V}\}$. Next, we will present some results of the Hausdorff measure of noncompactness in the space of regulated functions $\mathcal{R}(I, \mathbb{X})$.
\begin{lem}\cite{mesquita2012measure}
	Let $\mathcal{V}\subset\mathcal{R}(I, \mathbb{X})$ be bounded and equiregulated on $I$. Then
	\begin{enumerate}
		\item[(i)] $\lambda(\mathcal{V}(t))$ is regulated on $I$.
		\item[(ii)] $\lambda(\mathcal{V})=\sup\{\lambda(\mathcal{V}(t)): t\in I\}$.
	\end{enumerate}
\end{lem}
Let $\mathcal{LS}_{h}(I;\mathbb{X})$ be the space of functions $f:I\to \mathbb{X}$ that are Lebesgue-Stieltjes integrable with respect to $h$. Let $\mu_h$ be the Lebesgue-Stieltjes measure on $I$ induced by $h$. Using this measure, we have the following lemma.
\begin{lem}\cite[Corollary 3.1]{heinz1983behaviour}\label{Lemme 2.13}
	Let $\mathcal{V}_0\subset\mathcal{LS}_{h}(I;\mathbb{X})$ be a countable set. Assume that there exists a positive function $p\in\mathcal{LS}_{h}(I;\mathbb{R}^+)$ such that $\|v(t)\|\leq p(t),\ \mu_{h}$ - a.e. holds for all $v\in\mathcal{V}_0$. Then we have
	$$
	\lambda\Big(\int_0^a\mathcal{V}_0(t)\mathrm{d}h(t)\Big)\le 2\int_0^a\lambda(\mathcal{V}_0(t))\mathrm{d}h(t).
	$$
\end{lem}
\begin{definition}
	The function $\zeta\in \mathcal{R}(I,\mathbb{X})$ is said to be a mild solution of the system (\ref{sys}) on $I$ if it satisfies the following measure integral equation: 
	$$
	\zeta(t)=\W(t,0)(\zeta_0-g(\zeta))+\dint_0^t\, \W(t,s)\mathcal{V} u(s)\mathrm{d}s+\dint_0^t\, \W(t,s)\delta(s,\zeta(s))\mathrm{d}h(s),\;\;\; t\in I.
	$$
\end{definition}

\begin{definition}
    The nonlocal system \eqref{sys} is said to be controllable on the interval $I$, if for every $\zeta_0, \zeta_1\in\mathbb{X}$, there is a control function $u\in L^2(I; \mathbb{K})$ such that the mild solution $\zeta(\cdot)$ of \eqref{sys} satisfies $\zeta(a) + g(\zeta)=\zeta_1$.
\end{definition}
Our proof of controllability results heavily relies on the following fixed point theorem, a nonlinear alternative of M\"onch type.
\begin{theorem}\label{Monch}\cite{monch1980boundary}
	Let $\mathbb{X}$ be a Banach space, $\mathbb{W}$ an open subset of $\mathbb{X}$, $0\in \mathbb{W}$, and $\psi:\overline{\mathbb{W}}\to\mathbb{X}$ continuous. Also, assume that $\psi$ satisfies \begin{enumerate}
		\item[(i)] the M\"onch's conditions, $$\mathcal{N}\subset\overline{\mathbb{W}}\, \text{is countable}, \mathcal{N}\subset\overline{co}(\{0\}\cup\psi(\mathcal{N}))\Longrightarrow\, \mathcal{N}\, \text{is relatively compact and},$$ 
		\item[(ii)] the boundary condition  $$\zeta\in\overline{\mathbb{W}},\  \beta\in (0,1), \zeta=\beta \psi(\zeta)\Longrightarrow\ \zeta\notin \partial \mathbb{W}.$$
	\end{enumerate} 
	Then $\psi$ has a fixed point in $\overline{\mathbb{W}}$.
\end{theorem} 

\section{Main results}\label{sect3}
The prime aim of this section is to establish the controllability of the system (\ref{sys}) by employing the M\"onch fixed point theorem.
The following assumptions are required to build our main results.
\begin{enumerate}
	\item[$\textbf{(H1)}$]\, The resolvent operator $(t,s)\mapsto \W(t,s)$ of the linear part of (\ref{sys}) is norm-continuous in $\mathbb{X}$ for $0\leq s\leq t\leq a$ and $\sup\{\lVert \W(t,s)\rVert:\;0\leq s\leq t\leq a\}\leq L_1$
	\item[$\textbf{(H2)}$]\, The function $\delta:I\times \mathbb{X}\rightarrow \mathbb{X}$ satisfies:
	\begin{enumerate}
		\item[$(i)$]\, for a.e. $t\in I$ with respect to $\mu_h$, the function $\delta(t,\cdot):\mathbb{X}\rightarrow \mathbb{X}$ is continuous and for all $\zeta\in \mathbb{X}$, the function $\delta(.,\zeta):I\rightarrow \mathbb{X}$ is $\mu_h$-measurable;
		\item[$(ii)$]\, there exist a function $n\in \mathcal{LS}_h(I;\mathbb{R}^+)$ and a non-decreasing continuous function $\omega:\mathbb{R}^+\rightarrow\mathbb{R}^+$ such that 
		$\lVert\delta(t,\zeta)\rVert\leq n(t)\omega(\lVert \zeta\rVert)$ for all $\zeta\in\mathbb{X}$ and almost all $t\in I$ and $\displaystyle\lim_{k\rightarrow +\infty} \inf \frac{\omega(k)}{k}=\gamma<+\infty;$
		\item[$(iii)$]\, there exists a function $U\in \mathcal{LS}_h(I;\mathbb{R}^+)$ such that $\lambda(\delta(t,D))\leq U(t)\lambda(D)$ for almost all $t\in I$ and every bounded subset $D\subset \mathbb{X}$.
	\end{enumerate}
	\item[$\textbf{(H3)}$] $g:\mathcal{R}(I,\mathbb{X})\rightarrow \mathbb{X}$ is continuous and compact. And there exist the positive constants $c$ and $d$ such that $\lVert g(\zeta)\rVert\leq c\lVert \zeta\rVert_{\infty} + d$, \;\; for all $\zeta\in \mathcal{R}(I, \mathbb{X})$.
	\item[$\textbf{(H4)}$] The linear operator $Z:L^2(I,\mathbb{K})\rightarrow\mathbb{X}$ is defined by $$Zu=\displaystyle\int_0^a\,\W(a,s)\mathcal{V}u(s)\mathrm{d}s$$ such that
	\begin{enumerate}
		\item[$(i)$] $Z$ has an invertible operator $Z^{-1}$ which takes value in $L^2(I;\mathbb{K})/\ker(Z)$ and there exist positive constants $L_2$, $L_3$ such that
		$$\lVert\mathcal{V}\rVert\leq L_2,\;\;\;\lVert Z^{-1}\rVert\leq L_3 ;$$
		\item[$(ii)$] there exists $P_Z\in L^1(I;\mathbb{R^+})$ such that, for almost all $t\in I$ and any bounded set $Q\subset \mathbb{X}$,
		$$\lambda((Z^{-1}Q)(t))\leq P_Z(t)\lambda(Q). $$
	\end{enumerate}
\end{enumerate}
\begin{theorem}\label{main}
	Assume that hypotheses (\textbf{R}), \textbf{(E1)-(E2)}, and \textbf{(H1)-(H4)} are satisfied. Then the measure driven system (\ref{sys}) is controllable on $I$ provided that 
	\begin{equation}\label{cond1}
	c \big[ L_1+L_1L_2L_3a^{\frac{1}{2}}(1+L_1)\big]+L_1\gamma\big[1+L_1L_2L_3a^{\frac{1}{2}}\big]\int_0^a\,n(s)\mathrm{d}h(s)\leq 1.
	\end{equation}
	\begin{equation}\label{cond2}     \widetilde{L}:=\bigg(2L_1+4L_1^2L_2\int_0^a\,P_Z(s)\mathrm{d}s\bigg)\int_0^a\,L(s)\mathrm{d}h(s)<1.
	\end{equation}
\end{theorem}
\begin{proof}
	By using the condition $\textbf{(H4)}$-(i), for arbitrary function $\zeta\in \mathcal{R}(I;\mathbb{X})$, we define the control
	$$
	u_\zeta(t)=Z^{-1}\Bigg[\zeta_1-g(\zeta)-\W(a,0)(\zeta_0-g(\zeta))-\int_0^a\,\W(a,s)\delta(s,\zeta(s))\mathrm{d}h(s)\bigg](t).
	$$
	Our goal is to show when using this control that the operator $\psi:\mathcal{R}(I,\mathbb{X})\rightarrow\mathcal{R}(I,\mathbb{X})$ defined by
	$$
	(\psi\zeta)(t)=\W(t,0)(\zeta_0-g(\zeta))+\int_0^t\,\W(t,s)\mathcal{V} u(s)\mathrm{d}s+\int_0^t\,\W(t,s)\delta(s,\zeta(s))\mathrm{d}h(s)
	$$ has a fixed point $\zeta(\cdot)$. This fixed point is then a mild solution of the measure system (\ref{sys}). It is easy to verify that $\zeta(a)=(\psi\zeta)(a)=\zeta_1-g(\zeta)$
	which implies that the measure system (\ref{sys}) is controllable.
	
	As a result of the assumptions $\textbf{(H1)}$ and $\textbf{(H2)}$ the integral in the above formula is well defined. Next, we will prove that $\psi$ has a fixed point by using the M\"onch's fixed point theorem.
	Let $r>0$ and $B_r=\{\zeta\in \mathcal{R}(I,\mathbb{X}): \lVert\zeta\rVert_\infty<r\}$. For every positive number $r$, $B_r$ is clearly an open subset in $\mathcal{R}(I,\mathbb{X})$ and $0\in B_r$. Write $\overline{B_r}=\{\zeta\in \mathcal{R}(I,\mathbb{X}): \lVert\zeta\rVert_\infty\leq r\}$\, and $\psi(\overline{B_r})=\{\psi(\zeta):\;\;\zeta(\cdot)\in \overline{B_r}\}$. The remaining part is listed as:\\
	\textbf{Step 1}: There exists a positive number $r$ such that $\psi(\overline{B_r})\subseteq \overline{B_r}$.\\
	Suppose the contrary, then for all $r>0$ there is  $\zeta_r(\cdot)\in \overline{B_r}$ such that $\psi(\zeta_r)(\cdot)\notin\overline{B_r}$. Hence $\lVert\psi(\zeta_r)(t)\rVert>r$ for some $t\in I$. Using the hypotheses $(H1)-(H4)$, we have
	\begin{align*}
	r&<\lVert\psi(\zeta_r)(t)\rVert=\lVert \W(t,0)(\zeta_0-g(\zeta_r)) + \int_0^t\,\W(t,s)\mathcal{V}u(s)\mathrm{d}s+\int_0^t \W(t,s)\delta(s,\zeta_r(s))\mathrm{d}h(s)\rVert\\
	&\leq L_1(\lVert\zeta_0\rVert+\lVert g(\zeta_r)\rVert)+\int_0^t\,\lVert \W(t,s)\mathcal{V}u(s)\rVert\,\mathrm{d}s+\int_0^t\lVert \W(t,s)\delta(s,\zeta_r(s))\rVert\,\mathrm{d}h(s)\\ 
	&\leq L_1(\lVert\zeta_0\rVert+\lVert g(\zeta_r)\rVert)+L_1L_2\int_0^a\lVert u(s)\rVert \mathrm{d}s+L_1\int_0^a\lVert\delta(s,\zeta
	_r(s))\rVert\,\mathrm{d}h(s)\\
	&\leq L_1(\lVert\zeta_0\rVert+\lVert g(\zeta_r)\rVert)+L_1L_2a^{\frac{1}{2}}\lVert\, u_{\zeta_r}\rVert_{L^2}+L_1\omega(r)\int_0^a\,n(s)\mathrm{d}h(s).
	\end{align*}
	Note that 
	\begin{align}\label{estc}
	\lVert u_{\zeta_r}\rVert_{L^2}&=\bigg\lVert Z^{-1}\Bigg[\zeta_1-g(\zeta_r)-\W(a,0)(\zeta_0-g(\zeta_r))-\int_0^a\,\W(a,s)\delta(s,\zeta_r(s))\,\mathrm{d}h(s) \Bigg]\bigg\rVert_{L^2}\nonumber\\
	&\leq L_3\Bigg[\lVert\zeta_1\rVert+L_1\lVert\zeta_0\rVert+(1 + L_1)\lVert g(\zeta_r)\rVert+ L_1\omega(r)\int_0^a\,n(s)\mathrm{d}h(s) \Bigg],
	\end{align}
	from this inequality, we obtain that
	\begin{align*}
	r&<(L_1+L_1^2L_2L_3a^{\frac{1}{2}})\lVert\zeta_0\rVert+\Bigg[L_1+L_1L_2L_3a^{\frac{1}{2}}(1+L_1) \Bigg](cr+d)\\
	&+L_1L_2L_3a^{\frac{1}{2}}\lVert\zeta_1\rVert+L_1\omega(r)\Bigg[1+L_1L_2L_3a^{\frac{1}{2}}\Bigg]\int_0^a\,n(s)\mathrm{d}h(s).
	\end{align*}
	By dividing both sides by $r$ and taking the lower limit as $r\rightarrow +\infty$, we get
	$$
	c \big[ L_1+L_1L_2L_3a^{\frac{1}{2}}(1+L_1)\big]+L_1\gamma\big[1+L_1L_2L_3a^{\frac{1}{2}}\big]\int_0^a\,n(s)\mathrm{d}h(s)\geq 1. 
	$$
	It is a contradiction to $(\ref{cond1})$. Hence there is some positive number $r$ such that $\psi(\overline{B}_r)\subset\overline{B}_r$.
	
	\textbf{Step 2}:\,$\psi:\overline{B_r}\rightarrow\overline{B_r}$ satisfies the boundary condition given in point \textbf{(ii)} of Theorem \ref{Monch}.\\
	
	If $\zeta\in \overline{B_r}$, $\alpha\in (0,1)$ and $\zeta=\alpha\psi(\zeta)$, then by \textbf{Step 1}, we have
	$$
	\lVert\zeta\rVert=\lVert\alpha\psi(\zeta)\rVert<\lVert\psi(\zeta)\rVert\leq r.
	$$
	Therefore $\zeta\notin \partial B_r$.\\
	
	\textbf{Step 3:}
	$\psi(\overline{B_r})$ is equiregulated on $I$.\\
	For $t_0\in [0,a)$, we have
	\allowdisplaybreaks
	\begin{align*}
	\lVert\psi(\zeta)(t)-\psi(\zeta)(t_0^+)\rVert&=\bigg\lVert \W(t,0)(\zeta_0-g(\zeta))+\int_0^t\,\W(t,s)\mathcal{V} u_\zeta(s)\mathrm{d}s\\
	&\hspace{1cm}+\int_0^t\,\W(t,s)\delta(s,\zeta(s))\mathrm{d}h(s)-\W(t_0^+,0)(\zeta_0-g(\zeta))\\
	&\hspace{1cm}-\int_0^{t_{0^+}}\,\W(t_0^+,s)\mathcal{V} u_\zeta(s)\mathrm{d}s-\int_0^{t_{0^+}}\,\W(t_0^+,s)\delta(s,\zeta(s))\mathrm{d}h(s)\bigg\rVert\\
	& = \bigg\lVert\Big( \W(t,0)-\W(t_0^+,0)\Big)(\zeta_0-g(\zeta))+\int_0^{t_0^+}\W(t,s)\mathcal{V} u_\zeta(s)\mathrm{d}s\\
	&\hspace{1cm} + \int_{t_0^+}^t\W(t,s)\mathcal{V} u_\zeta(s)\mathrm{d}s + \int_0^{t_0^+}\,\W(t,s)\delta(s,\zeta(s))\mathrm{d}h(s)\\
	&\hspace{1cm} + \int_{t_0^+}^t\,\W(t,s)\delta(s,\zeta(s))\mathrm{d}h(s)-\int_0^{t_0^+}\W(t_0^+,s)\mathcal{V}u_\zeta(s)\mathrm{d}s\\
	&\hspace{1cm}-\int_0^{t_0^+}\,\W(t_0^+,s)\delta(s,\zeta(s))\mathrm{d}h(s)\bigg\rVert\\
	&\leq\bigg\lVert\Big(\W(t,0)-\W(t_0^+,0)\Big)\zeta_0\bigg\rVert + \bigg\lVert(\W(t,0)-\W(t_0^+,0)\Big)g(\zeta)\bigg\rVert\\
	& + \int_0^{t_0^+}\bigg\lVert\Big(\W(t,s)-\W(t_0^+,s)\Big)\mathcal{V}u_\zeta(s)\bigg\rVert\mathrm{d}s\\
	&+\int_0^{t_0^+}\bigg\lVert\Big(\W(t,s)-\W(t_0^+,s)\Big)\delta(s,\zeta(s))\bigg\rVert\mathrm{d}h(s)\\
	&+\int_{t_0^+}^t\bigg\lVert\W(t,s)\mathcal{V}u_\zeta
	(s)\bigg\rVert\mathrm{d}s\\
	&+\int_{t_0^+}^t\bigg\lVert \W(t,s)\delta(s,\zeta(s))\bigg\rVert\mathrm{d}h(s)\\
	&:= M_1+M_2+M_3+M_4+M_5 + M_6,
	\end{align*}
	where 
	\allowdisplaybreaks
	\begin{align*}
	&M_1 = \bigg\lVert\Big(\W(t,0)-\W(t_0^+,0)\Big)\zeta_0\bigg\rVert,\\
	&M_2 = \bigg\lVert(\W(t,0)-\W(t_0^+,0)\Big)g(\zeta)\bigg\rVert,\\
	&M_3 = \int_0^{t_0^+}\bigg\lVert\Big(\W(t,s)-\W(t_0^+,s)\Big)\mathcal{V}u_\zeta(s)\bigg\rVert\mathrm{d}s,\\
	&M_4 =  \int_0^{t_0^+}\bigg\lVert\Big(\W(t,s)-\W(t_0^+,s)\Big)\delta(s,\zeta(s))\bigg\rVert\mathrm{d}h(s),\\
	&M_5 = \int_{t_0^+}^t\bigg\lVert \W(t,s)\mathcal{V}u_\zeta
	(s)\bigg\rVert\mathrm{d}s,\\
	&M_6 = \int_{t_0^+}^t\bigg\lVert \W(t,s)\delta(s,\zeta(s))\bigg\rVert\mathrm{d}h(s).
	\end{align*}
	Therefore, we need to show that $M_i$ tends to $0$ independently of $\zeta\in\overline{B_r}$ when $t\rightarrow t_0$, $i=1,2,3,4,6$.\\
 Firstly for \textcolor{blue}{$t_0^+>0$}. By the norm-continuity of $\W(t,s)$, we have
	$$
	\lVert \W(t,0)\zeta_0-\W(t_0^+,0)\zeta_0\rVert\rightarrow 0\;\;\mbox{as}\;\;t\to t_0^+.
	$$
	Therefore, $M_1\rightarrow 0$.\\
	Moreover, we have that
	\begin{align*}
	\lVert \W(t,0)g(\zeta)-\W(t_0^+,0)g(\zeta)\rVert&\leq\lVert \W(t,0)-\W(t_0^+,0)\rVert\times\lVert g(\zeta)\rVert\\
	&\leq \lVert \W(t,0)-\W(t_0^+,0)\rVert\times \Big(c\lVert\zeta\rVert_{\infty}+d\Big)\\
	&\leq \lVert \W(t,0)-\W(t_0^+,0)\rVert\times \Big(cr+d\Big).
	\end{align*}
	And so, $\lVert \W(t,0)-\W(t_0^+,0)\rVert\rightarrow 0$ as $t\rightarrow t_0^+$  by the norm-continuity of $\W(\cdot,\cdot)$. Hence $M_2\rightarrow 0$ as $t\rightarrow t_0^+$.\\
	Now for $t_0^+ =0$,  we have that\\
	$\|\W(t,0)g(v) - g(v)\|\le  \sup_{\omega\in \overline{g(B_{r})}}\|\W(t,0)\omega - \omega\|\to 0, \mbox{as}\,\omega \to 0^+$, by Lemma \ref{uni} since $\overline{g({B_r})}$ is compact. Therefore, \,$I_2$ tends to $0$ as  $t\to t_0^+$.\\

	For $t_0^+=0$, $0<t\leq a$, it is easy to verify that $M_3=M_4=0$.\\
	For $0<t_0^+<a$ and arbitrary $0<\theta<t_0^+$, by the assumptions 
	$\textbf{(H1)}$ and $\textbf{(H4)(i)}$ and the arbitrariness of $\theta$, we get
	\begin{align*}
	M_3&\leq \displaystyle\int_0^{t_0^+-\theta}\bigg\lVert\Big(\W(t,s)-\W(t_0^+,s)\Big)\mathcal{V}u_\zeta\bigg\rVert\mathrm{d}s\\
	&+\int_{t_0^+-\theta}^{t_0^+}\bigg\lVert\Big(\W(t,s)-\W(t_0^+,s)\Big)\mathcal{V}u_\zeta\bigg\rVert\mathrm{d}s\\
	&\leq \displaystyle\sup_{s\in[0,t_0^+-\theta]}\bigg\lVert \W(t,s)-\W(t_0^+,s)\bigg\rVert_{\mathcal{L}(\mathbb{X})}\times L_2\int_0^{t_0^+-\theta}\lVert u_\zeta(s)\rVert\mathrm{d}s\\
	&+L_2\int_{t_0^+-\theta}^{t_0^+}\bigg\lVert\Big(\W(t,s)-\W(t_0^+,s)\Big)\bigg\rVert\lVert u_\zeta\rVert\mathrm{d}s\\
	&\rightarrow 0\; \mbox{as}\;t\rightarrow t_0^+\;\mbox{and}\;\theta\rightarrow 0.
	\end{align*}
	In the same way as previously, by using $\textbf{(H1)}$ and $\textbf{(H2)}$-(ii) we have
	\begin{align*}
	M_4&\leq \displaystyle\int_0^{t_0^+-\theta}\bigg\lVert\Big(\W(t,s)-\W(t_0^+,s)\Big)\delta(s,\zeta(s))\bigg\rVert\mathrm{d}h(s)\\
	&+\int_{t_0^+-\theta}^{t_0^+}\bigg\lVert\Big(\W(t,s)-\W(t_0^+,s)\Big)\delta(s,\zeta(s))\bigg\rVert\mathrm{d}h(s)\\
	&\leq \displaystyle\sup_{s\in[0,t_0^+-\theta]}\bigg\lVert \W(t,s)-\W(t_0^+,s)\bigg\rVert_{\mathcal{L}(\mathbb{X})}\int_0^{t_0^+-\theta}n(s)\omega(r)\mathrm{d}h(s)\\
	&+\int_{t_0^+-\theta}^{t_0^+}\bigg\lVert\Big(\W(t,s)-\W(t_0^+,s)\Big)n(s)\omega(r)\bigg\rVert\mathrm{d}h(s)\\
	&\leq \displaystyle\sup_{s\in[0,t_0^+-\theta]}\bigg\lVert \W(t,s)-\W(t_0^+,s)\bigg\rVert_{\mathcal{L}(\mathbb{X})}\omega(r)\int_0^{t_0^+-\theta}n(s)\mathrm{d}h(s)\\
	&+\omega(r)\int_{t_0^+-\theta}^{t_0^+}\bigg\lVert\Big(\W(t,s)-\W(t_0^+,s)\Big)n(s)\bigg\rVert\mathrm{d}h(s)\\
	&\rightarrow 0\;\mbox{as}\;t\rightarrow t_0^+\;\mbox{and}\;\theta\rightarrow 0.
	\end{align*}
	
	For $M_5$,  by $(H4)-(i)$  and (\ref{estc}), we get that 
	$M_5 = \dint_{t_0^+}^t\bigg\lVert \W(t,s)\mathcal{V}u_\zeta
	(s)\bigg\rVert\mathrm{d}s\;\;\; \rightarrow 0\;\;\;\mbox{as}\;\;\;t\rightarrow t_0^+.$\\
	
	Let $q(t)=\int_0^t\,n(s)\mathrm{d}h(s)$, from Lemma \ref{lemme1},\,$q(t)$ is a regulated function on $I$. Hence,
	$$
	M_6\leq L_1\omega(r)\int_{t_0^+}^t\,n(s)\,\mathrm{d}h(s)=L_1\omega(r)(q(t)-q(t_0^+))\rightarrow 0,\;\;\;\mbox{as}\;\;t\rightarrow t_0^+,
	$$
	also independently of $\zeta(\cdot)$.\\
	Therefore $\lVert\psi(\zeta)(t)-\psi(\zeta)(t_0^+)\rVert\rightarrow 0$ as $t\rightarrow t_0^+$.\\
	Using the similar method as previously, we show that $\lVert\psi(\zeta)(t_0^-)-\psi(\zeta)(t)\rVert\rightarrow 0$, as $t\rightarrow t_0^-$ for all $t_0\in(0,a]$. Hence, $\psi(\overline{B_r})$ is equiregulated on $I$.
	
	\textbf{Step 4 :} The operator $\psi$ is continuous on $\overline{B_r}$.\\
	Let $\{\zeta_n\}_{n=1}^\infty$ be a convergent sequence in $\overline{B_r}$ and $\zeta_n\rightarrow\zeta$ as $n\rightarrow\infty$. Using the hypotheses $\textbf{(H2)}$, $\textbf{(H3)}$ and the fact that $\W(t,s)$ is strongly continuous, we have for all $t\in I$,
	\allowdisplaybreaks
	\begin{align*}
	\lVert\psi(\zeta_n)(t)-\psi(\zeta)(t)\rVert&=\bigg\lVert-\W(t,0)(g(\zeta_n-g(\zeta))+\int_0^t\,\W(t,s)\mathcal{V}(u_{\zeta_n}(s)-u_{\zeta}(s))\,\mathrm{d}s\\
	&+\int_0^t\,\W(t,s)\big[\delta(s,\zeta_n(s))-\delta(s,\zeta(s)) \big]\mathrm{d}h(s)\bigg\rVert\\
	&\leq L_1\lVert g(\zeta_n)-g(\zeta)\rVert+L_1L_2\int_0^t\,\lVert u_{\zeta_n}(s)-u_\zeta(s)\rVert\,\mathrm{d}s\\
	&+L_1\int_0^t\lVert\delta(s,\zeta_n(s))-\delta(s,\zeta(s))\rVert \mathrm{d}h(s)\\
	&\leq L_1\lVert g(\zeta_n)-g(\zeta)\rVert+L_1L_2 a^{\frac{1}{2}}\lVert u_{\zeta_n}-u_\zeta\rVert_{L^2}\\
	& + L_1\int_0^t\lVert\delta(s,\zeta_n(s))-\delta(s,\zeta(s))\rVert \mathrm{d}h(s).
	\end{align*}
	Note that 
	\begin{align*}
	\lVert u_{\zeta_n}-u_\zeta\rVert_{L^2}&\leq L_3\bigg[\lVert g(\zeta_n)-g(\zeta)\rVert+L_1\lVert g(\zeta_n)-g(\zeta)\rVert\\
	&+L_1\int_0^a\,\lVert\delta(s,\zeta_n(s))-\delta(s,\zeta(s))\rVert\,\mathrm{d}h(s)\bigg].
	\end{align*}
	Then by hypotheses $\textbf{(H2)}$, $\textbf{(H3)}$ and the dominated convergence theorem for Lebesgue-Stieltjes integrals, we obtain that
	$$
	\lVert\psi(\zeta_n)(t)-\psi(\zeta)(t)\rVert\rightarrow 0\;\;\mbox{as}\;\;n\rightarrow\infty.
	$$
	Moreover, by \textbf{step 3} we know that $\{\psi(\zeta_n)\}_{n=1}^\infty$ is equiregulated . This result combined with the fact that $\psi(\zeta_n)\rightarrow\psi(\zeta)$ as $n\rightarrow\infty$ and Lemma \ref{lemm23} gives $\psi(\zeta_n)$ converges uniformly to $\psi(\zeta)$ as $n\rightarrow\infty$. Hence, $\psi$ is a continuous operator on $\overline{B_r}$.\\
	
	\textbf{Step 5 :} The M\"onch's condition holds.\\
	Suppose $D\subset \overline{B_r}$ is countable and $D\subset \overline{co}(\{0\}\cup\psi(D))$, we will show that $\lambda(D)=0$.\\
	Without loss of generality, we may suppose that $D=\{\zeta_n\}_{n=1}^\infty$. \textbf{Step 3} implies that $\psi(D)$ is equiregulated on $I$, then by using the fact that $D\subset \overline{co}(\{0\}\cup\psi(D))$ and Lemma \ref{lemme 2.10}, we get that $D$ is also equiregulated on $I$.\\
	By Lemma \ref{Lemme 2.13}, hypotheses $\textbf{(H2)}$-(iii) and $\textbf{(H4)}$-(ii), we have
	\begin{align*}
	\lambda(\{u_{\zeta_n}(s)\}_{n=1}^\infty)&\leq P_Z(s)\lambda\Bigg(\bigg[ \int_0^a \;\W(a,s)\delta(s,\zeta_n(s))\,\mathrm{d}h(s)\bigg]_{n=1}^\infty\Bigg)\\
	&\leq 2P_Z(s)L_1\int_0^a L(s)\lambda(D(s))\mathrm{d}h(s).
	\end{align*}
	Therefore
	\allowdisplaybreaks
	\begin{align*}
	\lambda((\psi D)(t)) &= \lambda(\{(\psi\zeta_n)(t)\}_{n=1}^\infty)\\
	&\leq \lambda(\{\W(t,0)(\zeta_0-g(\zeta_n)\}_{n=1}^\infty)+\lambda\bigg(\bigg[\int_0^t\;\W(t,s)\delta(s,\zeta_n(s))\,\mathrm{d}h(s)\bigg]_{n=1}^\infty\bigg)\\
	&+\lambda\bigg(\bigg[\int_0^t\;\W(t,s)\mathcal{V}u_{\zeta_n}(s)\,\mathrm{d}s\bigg]_{n=1}^\infty\bigg)\\
	&\le 2L_1\int_0^a\,L(s)\lambda(D(s))\,\mathrm{d}h(s)+2L_1L_2\int_0^a\lambda(\{u_{\zeta_n}(s)\}_{n=1}^\infty)\,\mathrm{d}s\\
	&\leq\bigg(2L_1+4L_1^2L_2\int_0^a\,P_Z(s)\,\mathrm{d}s\bigg)\int_0^a\,L(s)\lambda(D(s))\mathrm{d}h(s)\\
	&\leq\bigg(2L_1+4L_1^2L_2\int_0^a\,P_Z(s)\,\mathrm{d}s\bigg)\int_0^a\,L(s)\,\mathrm{d}h(s)\lambda(D)=\widetilde{L}\lambda(D).
	\end{align*}
	Using the previous result and the M\"onch condition, we obtain that
	$$
	\lambda(D)\leq\lambda(\overline{co}\{0\}\cup\psi(D)))=\lambda(\psi(D))\leq \widetilde{L}\lambda(D).
	$$
	Since $\widetilde{L}<1$, then $\lambda(D)=0$. Therefore, the set $D$ is relatively compact in $\mathcal{R}(I,\mathbb{X})$.\\
	
	According to Theorem \ref{Monch}, the operator $\psi$ has a fixed point in $\overline{B_r}$. Hence, the nonlocal measure-driven system (\ref{sys}) is  controllable on the interval $I$. 	
\end{proof}

\section{Example}\label{sect4}
Consider the integrodifferential systems driven by a measure of the form
\begin{equation}\label{exple}
\left\{\begin{array}{l c l}
\displaystyle d_tN(t,x)=\tau(t)\frac{\partial^2}{\partial  x^2}N(t,x) 
+ \int_0^tG(t,s)\frac{\partial^2}{\partial  x^2}N(s,x)\mathrm{d}s
+\theta\, \nu(t,x) \\ 
\hspace{3cm}+ M_0\cos(N(t,x))\,\mathrm{d}h(t),  \;\; t\in I=[0, 1],\;\;x\in[0,\pi],\\ 
N(t,0)=N(t,1)=0,\;\;t\in[0,1], \\ 
N(0,x)+\displaystyle\int_0^1\int_0^1\,f(t,x)\log(1+|N(s,p)|^{\frac{1}{2}})\,\mathrm{d}t\mathrm{d}p=N_0(x),\;\;x\in[0,1], 
\end{array}\right.
\end{equation}
where $\theta, M_0>0$ are constants, and $\tau: \mathbb{R}^+\to \mathbb{R}$ is a H\"older continuous function with order $0<\alpha\le 1$, meaning that there exists a positive constant $K_\tau$ such that 
$$
|\tau(t) - \tau(s)|\le K_\tau|t-s|^\alpha\,  \text{ for } t,s\in\mathbb{R}^+.
$$
Moreover, $\tau: \mathbb{R}^+\to \mathbb{R}$ is continuously differentiable and $\tau(t)<-1$.\\ 
The function $G(\cdot, \cdot)\in BU(\mathbb{R}^+\times\mathbb{R}^+,\mathbb{R})$, where $BU(\mathbb{R}^+\times\mathbb{R}^+, \mathbb{R})$ is the space of all bounded uniformly continuous functions from $\mathbb{R}^+\times \mathbb{R}$ into $\mathbb{R}$. In addition, $t\to\frac{\partial}{\partial t}G(t,s)$ is bounded and continuous from $\mathbb R^+$ to $\mathbb R$.\\ We introduce the space $\mathbb{X}=L^2([0,\pi], \mathbb{R})$ so that the system \eqref{exple} can be written in the abstract form \eqref{sys} and define the family of operators $\mathcal{A}(t)$ as follows: 
\begin{equation*}
    \begin{cases}
        (\mathcal{A}z)(x) = \tau(t)z^{\prime\prime}(x), t\ge 0, x\in [0,\pi],\\
        z\in \mathsf{D}(\mathcal{A}(t))= \mathsf{D}(\mathcal{A})=H^2(0,\pi)\cap H_0^1(0,\pi).
    \end{cases}
\end{equation*}
The operator $\mathcal{A}$ is defined as: 
\begin{equation*}
    \begin{cases}
        (\mathcal{A}z)(x) = z^{\prime\prime}(x), x\in [0,\pi],\\
        z\in \mathsf{D}(\mathcal{A}(t))= \mathsf{D}(\mathcal{A})=H^2(0,\pi)\cap H_0^1(0,\pi),\,  t\ge 0.
    \end{cases}
\end{equation*}
and 
\begin{equation*}
    \begin{cases}
        (\mathcal{B}(t,s)z)(x) = G(t,s)z^{\prime\prime}(x),\, 0\le s\le t, x\in [0,\pi],\\
        z\in \mathsf{D}(\mathcal{A}(t))= \mathsf{D}(\mathcal{A})=H^2(0,\pi)\cap H_0^1(0,\pi).
    \end{cases}
\end{equation*}


In what follows we will show that the following assumptions are verified. \\
\begin{itemize}
    \item[---] \textbf{Assumptions (C1)-(C3):}
\begin{itemize}
    \item Assumptions (C1): for each $t\ge 0, \mathcal{A}(t)=\tau(t)\Delta$ generates a strongly continuous semigroup on $\mathbb{X}$. In addition, $\mathcal{A}(t)z=\tau(t)\Delta z$ is strongly continuously differentiable on $\mathbb{R}^+$ for each $z\in\mathsf{D}(\mathcal{A})$, due to the strong continuous differentiability of  $\tau(\cdot)$. We known that $\Delta$ is an infinitesimal generator of a $C_0$-semigroup of contraction on $\mathcal{X}$. Since $b(t) <-1$, it follows that $\tau(t)\Delta$ is an infinitesimal generator of a $C_0$-semigroup of contraction on $\mathbb{X}$. Therefore, for each $t\ge 0, \mathcal{A}(t)=\tau(t)\Delta$ is stable (see, \cite{pazy2012semigroups}). Furthermore, we have $(\mathcal{B}(t)z)(\cdot)=G(t +\cdot, t)\mathcal{A}z$ for each $z\in\mathsf{D}(\mathcal{A})$. Since $\frac{\partial G(t,s)}{\partial t}$ is continuously differentiable from $\mathbb{R}^+$ to $\mathbb{R}$, then $\mathcal{B}(t)z$ is strongly continuously differentiable on $\mathbb{R}^+$ for each $z\in\Y$.  Hence, assumption \textbf{(C1)} is satisfied. 
    
   Thus for $z\in\mathsf{D}(\mathcal{A})=\mathsf{D}$, 
$$ \mathcal{A}(t)z = \sum_{n=1}^{+\infty}(-n^2\tau(t))\langle z, e_n\rangle e_n$$
where $\langle\cdot, \cdot\rangle$ is the usual inner product on $\mathbb{X}$.
Therefore, $\mathcal{A}(t)$ generate an evolution family $\{S(t,s): 0\le t\le 1\}$ given by:
\begin{equation}\label{semie}
    S(t,s)z=\sum_{n=1}^{+\infty} e^{((-n^2)(t-s)\times \int_s^t\tau(u)du)}\langle z, e_n\rangle e_n, \text{ for } 0\le s\le t\le T, z\in \mathbb{X}.
\end{equation}
we show that the evolution family $\{S(t,s): 0\le s\le t\}$ is norm-continuous for $t-s>0$.\\
Let $0<s_1<s_2<t_1<t_2$,
\allowdisplaybreaks
\begin{align*}
\lVert S(t_2,s_2)z-S(t_1,s_1)z\rVert &=\bigg\lVert\sum_{n=1}^{+\infty} e^{((-n^2)(t_2-s_2)\times \int_{s_2}^{t_2}\tau(u)du)}\langle z, e_n\rangle e_n\\ &-\sum_{n=1}^{+\infty} e^{((-n^2)(t_1-s_1)\times \int_{s_1}^{t_1}\tau(u)du)}\langle z, e_n\rangle e_n\bigg\rVert\\
&=\bigg\lVert\sum_{n=1}^{+\infty} e^{((-n^2)(t_2-s_2)\times (\int_{s_2}^{t_1}\tau(u)du+ \int_{t_1}^{t_2}\tau(u)du))}\langle z, e_n\rangle e_n\\
&-\sum_{n=1}^{+\infty} e^{((-n^2)(t_1-s_1) \times(\int_{s_1}^{s_2}\tau(u)du+ \int_{s_2}^{t_1}\tau(u)du))}\langle z, e_n\rangle e_n\bigg\rVert
\end{align*}
Since the series $\sum_{n=1}^{+\infty} e^{((-n^2)(t-s)\times\int_s^t\tau(u)du)}\langle z, e_n\rangle e_n, \text{ for } 0\le s\le t\le 1, z\in \mathbb{X}$ is convergent then if $t_2\rightarrow t_1$ and $s_2\rightarrow s_1$, we obtain
$\lVert S(t_2,s_2)z-S(t_1,s_1)z\rVert \rightarrow 0$. Therefore, $\{S(t,s): 0\le s\le t\le 1\}$ is norm-continuous for $t-s>0$

    \item Assumptions (C2): Since $G(\cdot,\cdot)$ is continuous, then $(\mathcal{B}(t))(s)=G(t+s, t)$ is continuous on $\mathbb{R}^+$. Moreover, for any $t\geq 0$ and any $y\in \Y$, we have:
\begin{align*} \|\mathcal{B}(t)y\|_\mathbb{X} &= \sup_{s\in \mathbb R^+}\|(\mathcal{B}(t) y)(s)\|_\mathbb{X}\\
&= \sup_{s\in \mathbb R^+}|(\mathcal{B}(t+s,t)|\|\mathcal{A}x\|_\mathbb{X}\\
&\le \sup_{s\in \mathbb R^+}|(\mathcal{B}(t+s,t)|(\|y\| + \|\tilde{A}y\|).
\end{align*}
Thus, assumption \textbf{(C2)} holds.
\item  Assumptions (C3): Let $y\in\Y$, then $\frac{d}{ds}(\mathcal{B}(t)y)(s)=\frac{\partial}{\partial s}G(t+s, t)\mathcal{A}y$, which implies that $\mathcal{B}(t)y\in D_s(\mathsf{D})$ and $$
\bigg(\mathsf{D}\mathcal{B}(t)y\bigg)(s) = \frac{\partial}{\partial s}G(t+s,t)\mathcal{A}y.
$$ Moreover, $t\to \frac{\partial}{\partial s}\mathcal{B}(t)$ is continuous on $\mathbb{R}^+$ and 
\begin{align*}
\|\mathsf{D}\mathcal{B}(t)y\|_\mathbb{X}\leq \sup_{s\in\mathbb{R}^+}|\frac{\partial}{\partial s}G(t+s, t)(\|y\| + \|\mathcal{A}y\|).\end{align*}
Hence, assumption \textbf{(C3)} holds, as $\frac{\partial}{\partial s}G(t+s, t)\mathcal{A}y\in\mathcal{L}(\Y,\mathfrak{F})$.
\end{itemize}

\item[---] Assumption \textbf{(E1)-(E3)}:
\begin{itemize}
    \item \textbf{(E1)} follows from \textbf{(C1)}. Assumption \textbf{(E2)} and \textbf{(E3)} hold for $M_1=1, M_2=\frac{K_\tau}{3}$, and $\theta=\alpha$. See \cite{arora2021approximate} for more details. Thus, there exists a unique evolution system $\{S(t,s)\}$ generated by the stable family $(\mathcal{A}(t), t\ge 0)$, which is given by \eqref{semie}.
Since the evolution system $\{S(t,s):0\le s\le t\le 1\}$ generated by $\mathcal{A}(t)$ is norm-continuous, then by Theorem \ref{NCR}, the corresponding resolvent operator $\W(t, s)$ of the linear part of (\ref{sys}) is norm continuous in $\mathbb{X}$ for $0<s<t<1$. Assumption \textbf{(H1)} is satisfied.
\end{itemize}
\end{itemize}
To complete the abstract formulation of \eqref{exple}, we introduce the following notations:
\begin{displaymath}
\left\{
\begin{array}{ll}
\zeta(t)(x)= N(t, x),\quad \zeta^\prime(t)(x)=d_t N(t,x)  \quad\mbox{for}\;\;  t\in[0,1]\,\mbox{ and }\, x\in [0,1],\\
\zeta(0)(x)=N(0, x),\quad\mbox{ for }\, x\in [0,1].
\end{array}
\right.
\end{displaymath}
Furthermore we define the maps $\delta:I\times\mathbb{X}\to\mathbb{X}$, $g:\mathcal{R}(I, \mathbb{X})\to \mathbb{X}$, $\mathcal{V}:\mathbb{K}\to\mathbb{X}$ by:
\begin{align*}
\delta(t,\zeta)(x)&=M_0\cos(\zeta(x)),\;\mbox{for}\;t\in[0,1],\;x\in[0,1],\;\zeta\in\mathbb{X},\\
g(\zeta)(x)&=\displaystyle\int_0^1\int_0^1\,f(t,x)\log(1+|\zeta(t)(p)|^{\frac{1}{2}})\,\mathrm{d}t\mathrm{d}p,\\
\mathcal{V}u(t)(x)&=u(t)(x)=\theta\, \nu(t,x),\; x\in [0,1].
\end{align*}
Using these definitions we can represent the  system \eqref{exple}  in the following abstract form 
\begin{equation}\label{sysexple}
\left\{\begin{array}{l c l}
\mathrm{d}\zeta(t)= \left[A(t)\zeta(t) + \displaystyle\int_{0}^{t}\Delta(t,s)\zeta(s)\mathrm{d}s + \mathcal{V}u(t)\right]\mathrm{d}t  + \delta(t,\zeta(t))\mathrm{d}h(t),\ t\in I=[0,a], a>0,\\ 
\zeta(0) + g(\zeta)=\zeta_0.&\\
\end{array}\right.
\end{equation}
Suppose that $h:[0, 1]\rightarrow\mathbb{R}$ is defined by
\begin{equation*}
h(t)=\left\{\begin{array}{l c l}
1-\frac{1}{2},\qquad 0\leq t\leq 1-\frac{1}{2},  \\ 
\cdots  \\ 
1-\frac{1}{k},\qquad 1-\frac{1}{k-1}<t\leq 1-\frac{1}{k},\;\;\mbox{ for }\; k>2\;\;\mbox{ and }\;\;k\in\mathbb{N}, \\ 
\cdots  \\ 
1,\qquad t=1.
\end{array} \right.
\end{equation*}
It is clear that $h:[0,1]\rightarrow\mathbb{R}$ is a left continuous (at $t=1$) and nondecreasing function on $[0,1]$.\\
We assume that 
\begin{enumerate}
	\item[(a)]$G:[0,1]\times \mathbb{X}\rightarrow\mathbb{X}$  is a continuous function defined by
	$$
	G(t,\zeta)(x)=\delta(t,\zeta(t,x)),\;\;\; t\in [0,1],\;\;x\in[0,1].
	$$
	Take  $\delta(t,\zeta(t,x))=M_0\cos(\zeta(x))$,  where  $M_0$ is a positive constant. Then  $G$ is Lipschitz continuous with respect to the second variable because
	\begin{align*}
	\lVert G(t,\zeta_2)-G(t,\zeta_1)\rVert&=\bigg(\int_0^1\,M_0^2|\cos(\zeta_2(x))-\cos(\zeta_1(x))|^2\,\mathrm{d}x\bigg)^{\frac{1}{2}}\\
	&\leq M_0\bigg(\int_0^1\,|\zeta_2(x)-\zeta_1(x)|^2\,\mathrm{d}x\bigg)^{\frac{1}{2}}=M_0\lVert \zeta_2 - \zeta_1\rVert.
	\end{align*}
	Moreover, since 
	$$
	\lVert G(t,\zeta)\rVert=\bigg(\int_0^1\,M_0^2\cos^2(\zeta(x))\,dx\bigg)^{\frac{1}{2}}\leq \bigg(\int_0^1\,M_0^2\,\zeta^2(x)\,dx\bigg)^{\frac{1}{2}}=M_0\lVert \zeta\rVert,
	$$
	then we can  take $n(t)=M_0$, $\omega(t)=t$, and hence $\gamma=1$ in the hypothesis $\textbf{(H2)}$-(ii). Moreover, since $\lambda(\delta(t,D))\leq M_0\lambda(D)$ for any bounded subset $D$ of $\mathbb{X}$ then $U(t)=M_0$. Hence, $G$ verifies hypothesis $\textbf{(H2)}$ in Theorem \ref{main}.
	\item[(b)] $g:\mathcal{R}([0,1],\mathbb{X})\rightarrow\mathbb{X}$ is a continuous defined by 
	$$
	g(\zeta)(x)=\displaystyle\int_0^1\int_0^1\,f(t,x)\log(1+|\zeta(t)(p)|^{\frac{1}{2}})\,\mathrm{d}t\mathrm{d}p,\qquad \zeta\in\mathcal{R}([0,1];\mathbb{X}).
	$$
	We can use the same proof as in the literature \cite{cao2016measures} to show that $g$ is a continuous and compact operator and we have
	\begin{align*}
	\lVert g(\zeta)\rVert _\mathbb{X}&=\bigg\lVert \displaystyle\int_0^1\int_0^1\,f(t,x)\log(1+|\zeta(t)(p)|^{\frac{1}{2}})\,\mathrm{d}t\mathrm{d}p\bigg\rVert_\mathbb{X}\\
	&\le \max_{t\in[0,1],x\in[0,1]}\mid f(t,x)\mid[ \Vert\zeta\Vert_{\infty}+1].
	\end{align*}
	Therefore the hypothesis \textbf{(H3)} in Theorem \ref{main} is satisfied with \\
	$c=\max_{t\in[0,1],x\in[0,1]}|f(t,x)|$ and $d=1$.
	\item[(c)] Let the control operator $\mathcal{V}:\mathbb{X}\rightarrow \mathbb{X}$ be defined by
	$$
	(\mathcal{V}u)(t)(x)=\theta\,\nu(t,x),\qquad x\in [0,1].
	$$
	It is easy to verify that $\lVert\mathcal{V}\rVert=\theta$. 
	The linear operator $Z:L^2([0,1],\mathbb{X})\rightarrow \mathbb{X}$ is given by
	$$
	Zu=\int_0^1\,\W(1-s)\theta\,u(s,x)\,\mathrm{d}s,
	$$
	It is easy to verify that $\lVert Z\rVert\le\theta$. By Appendix of \cite{Bala} and the inverse operator theorem, $Z$ has a bounded inverse operator $Z^{-1}$ taken value in $L^2([0,1];\mathbb{X})/KerZ$.\\
	According to the above analysis, let $L_2=\theta$, $c=\max_{t\in[0,1],x\in[0,1]}\mid f(t,x)\mid$, then the inequalities (\ref{cond1}) and (\ref{cond2}) take the form
	\begin{align}\label{eq5}
	\big[ L_1+L_1\theta L_3(1+L_1)\big]\lVert f\rVert_{L^2}+\frac{1}{2}M_0L_1(1+L_1\theta L_3)< 1.
	\end{align}
	and 
	\begin{align}\label{eq6}
	M_0\bigg(L_1+2L_1^2\theta\int_0^1\,P_Z(s)\,\mathrm{d}s\bigg)<1
	\end{align}
	respectively.\\
	Further, suppose the condition $\textbf{(H4)}$-(ii) holds and the above inequalities \eqref{eq5}, \eqref{eq6}  are satisfied, then the integrodifferential equation (\ref{exple}) is exactly nonlocally controllable on $[0,1]$ by Theorem \ref{main}.
\end{enumerate}
\section{Conclusion}
This work investigates a category of nonlocal measure-driven integral differential systems in Banach space. Our results, proven using the M\"onch fixed point theorem, which does not impose the compactness of the resolvent operator, represent a novel contribution to this field. Our next step will be to apply this finding to a nonlocal evolution equation with state-dependent delay, non-instantaneous impulsive, and random effects.


\section*{\textbf{Acknowledgments}}
We are grateful to anonymous referees and the editor for their constructive suggestions, which improve the quality of this manuscript. 

\end{document}